\documentclass[12pt]{amsart}

\usepackage{xcolor}
\usepackage{amsmath,amsthm,amsfonts,amscd,amssymb,amsopn,enumerate}
\usepackage{eucal, mathrsfs,color,pxfonts}
\usepackage[all]{xy}
\usepackage{fullpage}

\newcommand{\mbb}[1]{\mathbb #1}

\newcommand{\ms}[1]{\mathscr #1}

\newcommand{\bound}{d}
\def\<{\left<}
\def\>{\right>}

\newcommand{\PP}{\mathbb P}
\newcommand{\ZZ}{\mathbb Z}

\newcommand{\ov}[1]{\overline{#1}}
\newcommand{\til}[1]{\widetilde{#1}}

\DeclareMathOperator{\Spec}{Spec}
\DeclareMathOperator{\per}{per}
\DeclareMathOperator{\ind}{ind}

\DeclareMathOperator{\id}{id}

\theoremstyle{plain}

\newtheorem{thm}{Theorem}[section]
\newtheorem{defn}[thm]{Definition}

\newtheorem{lem}[thm]{Lemma}
\newtheorem{cor}[thm]{Corollary}
\newtheorem{prop}[thm]{Proposition}
\newtheorem*{thm*}{Theorem}
\newtheorem*{rem*}{Remark}
\newtheorem*{lem*}{Lemma}

\newtheorem*{cor*}{Corollary}
\newtheorem*{prop*}{Proposition}

\newtheorem{rem}[thm]{Remark}

\title[Torsion in Chow groups of zero cycles]{Torsion in Chow groups of
zero cycles of homogeneous projective varieties}

\author{Daniel Krashen}

\begin{document}


\maketitle

\begin{abstract}
We give bounds on the order of torsion in the Chow group of
zero dimensional cycles for isotropic Grassmannians and
Brauer-Severi flag varieties. To do this, we introduce tools to understand
the behavior of torsion in Chow groups with coefficients under morphisms of
proper varieties
\end{abstract}

\section{Introduction}

The problem of understanding torsion in the Chow group of projective
homogeneous varieties has had rich and surprising applications. A
topic of much investigation, it has given important insight into the study
of central simple algebras and quadratic forms (see \cite{BZZ,
Kar:TC2, Kar:TQuad, Kar:CGS, KarMer:CGPQ, Mer:H1K2}). The Chow group of
zero dimensional cycles plays an important role both due to its
connection with rational points, and because of its connections to the
structure of Chow groups in other dimensions (for example, see
\cite{KarMer:SNV}).

In the last few years significant progress has been made by showing that
for certain classes of projective homogeneous varieties, the Chow groups of
zero dimensional cycles are torsion free \cite{Kra:ZC, PSZ, ChMe:ZC}.
Despite these results, many gaps in our understanding still remain. 

In this paper, we introduce an alternative approach, relying on Rost's
theory of Chow groups with coefficients as developed in \cite{Rost:CGC}, in
order to obtain results for new classes of projective homogeneous
varieties. Instead of proving the non-existence of torsion for zero
dimensional cycles, we give bounds on the order of torsion.

\bigskip

Given a geometrically rational variety $X$ defined over a field $k$, the
Chow group of zero dimensional cycles of degree $0$ is always a torsion
group. More precisely, it is not hard to see by a standard
restriction-corestriction argument, that if $X_L$ is a rational variety for
$L/k$ a finite field extension, then the order of every element of this
group divides $[L:k]$. On the other hand, in many cases, we know that these
bounds on the torsion are not sharp. In fact, for many classes of
projective homogeneous varieties, it is shown in \cite{ChMe:ZC,Kra:ZC,
PSZ}, that this group is trivial in many cases despite the varieties
themselves being far from rational. It is natural to ask, therefore, in
which cases does this group possess nontrivial torsion, and how can one give
bounds on its order?

In this paper, we show for certain classes of varieties, namely isotropic
Grassmannians of quadratic forms, and Brauer-Severi flag varieties for
algebras of period $2$, one may consistently achieve better bounds on
torsion than would be implied by the above restriction-corestriction
argument.

 It is also an interesting and open question to
understand to what extent torsion in Chow groups of zero dimensional cycles
on one homogeneous variety contributes to the torsion in the Chow groups of
higher dimensional cycles of other homogeneous varieties via natural Chow
correspondences.

\section*{Acknowledgments}

The impetus for the writing of this paper came out of conversations between
myself and Kirill Zainouilline during the Thematic Program on Torsors,
Nonassociative Algebras and Cohomological Invariants held at the Fields
institute and the University of Ottawa. Many thanks also go to my father,
Stephen Krashen, for a generous supply of coffee during the writing of this
paper. This research was partially supported by NSF grants DMS-1007462 and
DMS-1151252.

\section{Notation and Preliminaries}

Let $k$ be a field. We let $K_\bullet ^M(k)$ denote the Milnor $K$-theory
ring of $k$. We will make frequent use of the machinery of Rost's ``Chow
groups with coefficients,'' as defined in \cite{Rost:CGC}.  We recall in
particular, that a cycle module $M$ (see \cite[Def.2.1]{Rost:CGC}), defines for
every field extension $L/k$, a graded Abelian group $M(L) = M_\bullet(L)$,
which is a graded module for the Milnor $K$-theory ring $K_\bullet ^M (k)$.
These are also equipped with, restriction maps $r_{L/k}: M(k) \to M(L)$ for
every field extension $L/k$ and for \textit{finite} field extensions $L/k$,
corestriction maps $c_{L/k}: M(L) \to M(k)$ which are compatible with the
corresponding maps in Milnor $K$-theory (see \cite[Def.~1.1]{Rost:CGC}). 

For a cycle module $M$ and a variety $X$, setting 
\[C_p(X; M, q) = \coprod_{x \in X_{(p)}} M_q(k(x)),\] 
we obtain a complex
\[\xymatrix{
 \ar[r] & C_{p+1}(X; M, q+1) \ar[r]^-{\bound_X} & C_p(X; M, q)
 \ar[r]^-{\bound_X} &  C_{p-1}(X; M, q-1) \ar[r] &
}\]
Following \cite[(3.2)]{Rost:CGC}, we denote the kernel of the map $C_p(X;
M, q) \to C_{p-1}(X; M, q-1)$ by $Z_p(X; M, q)$, and the homology at the
middle term by $A_p(X; M, q)$. For a class $\alpha \in Z_p(X; M, q)$, we
let $[\alpha]$ denote its class in $A_p(X; M, q)$. 

We refer to a pair $(z, m)$ where $z \in X_{(p)}$ and $m \in M_q(z)$ as a
\textit{prime chain}, and identify it with $m \in M_q(z) \subset \coprod_{x
\in X_{(p)}} M_q(k(x)) = C_p(X; M, q)$. If $\bound_X (z, m) = 0$, we refer to
it also as a \textit{prime cycle}. We will occasionally also need to use
the notation $(z, m)_X$ if it is necessary to make the variety $X$
explicit.

We recall that the functors $C_p(\mbox{\ \ \ } ; M, q)$, $Z_p(\mbox{\ \ \ } ;
M, q)$, and $A_p(\mbox{\ \ \ } ; M, q)$ are covariant with respect to
proper maps. If $\pi : X \to Y$ is proper, then we denote all the
corresponding pushforward maps by $\pi_*$. For a prime chain $(x, m) \in
C_p(X; M, q)$, these are defined by $\pi_*(x, m) = 0$ if $k(x)/k(\pi(x))$
is not a finite field extension, and otherwise, we set $\pi_*(x, m) =
(\pi(x), c_{k(x)/k(\pi(x))}(m))$.

\begin{defn}
Let $X$ be a proper variety.  We let $\ov{A}_0(X; M, q)$ denote the kernel
of the map $A_0(X; M, q) \to A_0(\Spec k; M, q) = M_q(k)$ induced by the
pushforward for the structure map $X \to \Spec k$.
\end{defn}

\begin{defn}
Let $X$ be a variety over $k$. We write $A_p(X, \mbb Z)$ for $A_p(X;
K_\bullet^M, 0) = CH_p(X)$, and $\ov A_0(X, \mbb Z)$ for $\ov A_0(X;
K_\bullet^M, 0)$.
\end{defn}
We note that this notation differs somewhat from that used in
\cite{Kra:ZC}, but is more consistent with the notation of \cite{Rost:CGC}.

\begin{defn}
Let $M$ be a cycle module over $k$, let $F/k$ be an arbitrary field
extension, and let $X/F$ be a variety. Let $M_F$ denote the cycle module
restricted to the extension fields of $F$. We define $t(X; M, q)$ to be
elements of $K^M_\bullet(k)$ which annihilate $\ov{A}_0(X; M_F, q)$. That
is:
\[t(X; M, q) = ann_{K^M_\bullet(k)} \ov{A}_0(X; M_F, q).\]
We set $t_0(X; M, q) = t(X; M, q) \cap K^M_0 = t(X; M, q) \cap \ZZ$.
\end{defn}

For a morphism $\pi: X \to Y$, and $W \to Y$ a morphism, we let $X_W$
denote the fiber product $X \times_Y W$.

\begin{defn}
Let $\pi : X \to Y$ be a proper morphism. We set $t(\pi_{(p)}, M)$
(respectively $t_0(\pi_{(p)}, M)$) to be the intersection of the ideals
$t(X_y, M_{k(y)})$ (resp. $t_0(X_y, M_{k(y)})$) as $y$ ranges over all the
points of $Y$ of dimension $p$.
\end{defn}

\begin{defn}
Let $X$ be a variety over $k$. We define $\ind(X)$ to be the greatest
common divisor of the degrees of all finite field extensions $L/k$ such
that $X(L) \neq \emptyset$.
\end{defn}

\begin{defn}
Let $\pi : X \to Y$ be a morphism. We define $\ind_p(\pi)$ to be the least
common multiple of all integers of the form $\ind(X_y)$ over $y \in
Y_{(p)}$.
\end{defn}

\section{Behavior of torsion in Chow groups under morphisms}

\begin{lem} \label{multiple lifts}
Let $\pi: X \to Y$ be a morphism. Then 
\[(\ind_p(\pi)) C_p(Y; M, q) \subset \pi_* C_p(X; M, q).\]
\end{lem}
\begin{proof}
To prove the statement, it is enough to show that
for every prime chain (i.e. each chain consisting of a single term) $(y,
m)$, $y \in Y_{(p)}$ and $m \in M_q(k(y))$, that $\ind_p(\pi)(y, m) \in
\pi_* C_p(X; M, q)$. By definition of the index, we may find a collection of
points $x_1, \ldots, x_r \in X_{(p)}$ lying over $y$ such that $x_i \cong
Spec(L_i)$ and $\gcd\{[L_i: k(y)]\} | \ind_p(\pi)(y, m)$. We will show that
$[L_i : k(y)] (y, m) \in \pi_* C_p(X; M, q)$, which will complete the proof.

To do this, we simply consider $(x_i, r_{L_i/k(y)}(m)) \in C_p(X; M, q)$.
By definition of the map $\pi_*$, we find $\pi_* (x_i, r_{L_i/k(y)}(m)) =
(y, [L_i:k(y)] m) = [L_i:k(y)] (y, m)$ as desired.
\end{proof}

\begin{lem} \label{trivial pushforward}
If $\alpha \in \ker\left( \pi_* : A_p(X; M, q) \to A_p(Y; M, q) \right)$,
then there exists $\til \alpha \in Z_p(X; M, q)$ such that $[\til \alpha] =
(\ind_{p+1} \pi)\alpha$ and 
\[\til \alpha \in \ker\left( \pi_*: Z_p(X; M, q) \to Z_p(Y; M, q)
\right).\]
\end{lem}
\begin{proof}
Choose a representative $\til \beta$ for $\alpha$ in $Z_p(X; M, q)$. By
definition, since the class of $\pi_* \til \beta$ is $0$, we may find
$\gamma \in C_{p+1}(Y; M, q+1)$ such that $\bound_Y \gamma = \pi_* \til
\beta$. By Lemma~\ref{multiple lifts}, we may find a cycle $\gamma' \in
C_{p+1}(X; M, q+1)$ such that $\pi_* \gamma' = (\ind_{p+1}) \gamma$.
Defining $\til \alpha = (\ind_{p+1}) \til \beta - \bound_X \gamma'$, we
then have that $[\til \alpha] = [\til \beta] = \alpha$, and by \cite[Prop.
4.6(1)]{Rost:CGC}, we have
\[ \pi_* \til \alpha = (\ind_{p+1}) \pi_* \til \beta - \pi_* \bound_X
\gamma' = (\ind_{p+1}) \pi_* \til \beta - (\ind_{p+1})\bound_Y \gamma =
0,\]
as desired.
\end{proof}

\begin{lem} \label{zero cycles trivial pushforward}
Suppose we have a morphism of proper varieties $\pi: X \to Y$, and $\alpha
\in Z_0(X; M, q)$ such that $\pi_* \alpha = 0$ in $Z_0(X; M, q)$. Then we
may find a reduced scheme $W$, finite over $k$, and an inclusion $j : W \to
Y$, such that if we consider the fiber product diagram:
\[\xymatrix{
X_W \ar[r]^i \ar[d]_{\pi_W} & X \ar[d]^\pi \\
W \ar[r]_j & Y,
}\]
then we have $\alpha = i_* \beta$ for some $\beta \in \ker\left((\pi_W)_* :
Z_0(X_W; M, q) \to Z_0(W; M, q)\right)$.
\end{lem}
\begin{proof}
We note that by the definition of a cycle complex \cite[Def.
3.2]{Rost:CGC}, it follows immediately that the pushforward is injective on
the level of complexes for closed immersions. In fact, we may naturally
identify the cycle complex of a closed subscheme with a subcomplex of the
cycle complex of the overscheme. Consequently, it suffices to show that we
may find $\beta \in C_0(X_W; M, q)$ with $i_* \beta = \alpha$, and the
final conclusion will follow automatically. But this now follows by setting
$W$ to be the image of the support of the $0$-dimensional chain $\alpha$.
\end{proof}

\begin{lem} \label{zero cycles pushforward again}
Let $\pi : X \to Y$ be a morphism of proper $k$-varieties. Then we have
\[ \left[\ker\left(\pi_*: A_0(X; M, q) \to A_0(Y; M, q)\right) \right]
(\ind_1 \pi) t(\pi_{(0)}; M, q) = 0.\]
\end{lem}
\begin{proof}
Let $\alpha \in \ker\left(\pi_*: A_0(X; M, q) \to A_0(Y; M, q)\right)$. By
Lemma~\ref{trivial pushforward}, we may find $\til \alpha \in Z_0(X; M, q)$
with $[\til \alpha] = (\ind_1 \pi) \alpha$ and $\pi_* \til \alpha = 0$. By
Lemma~\ref{zero cycles trivial pushforward}, we may find a reduced scheme
$W$, finite over $k$, and an inclusion $j : W \to Y$, such that if we set
$i : X_W \to X$ to be the natural fiber product map, we have $\til \alpha =
i_* \til \beta$ for some $\til \beta \in Z_0(X_W; M, q)$. From the
definition of $t$, it follows that we have $t(\pi_{(0)}, M) [\til \alpha] =
0$.  All together this tells us:
\[ t(\pi_{(0)}, M) (\ind_1 \pi) \alpha = t(\pi_{(0)}, M) [\til \alpha] = 0,
\]
as desired.
\end{proof}

\begin{prop} \label{fiber induction proposition}
Let $\pi : X \to Y$ be a morphism of proper varieties over $k$. Then
\[t(Y; M, q) t(\pi_{(0)}; M, q) \ind_1(\pi) \subset t(X; M, q), \text{ and
} \ind_0(\pi) t(X; M, q) \subset  t(Y; M, q)\]
\end{prop}
\begin{proof}
Let $\alpha \in \ov A_0(X; M, q)$. We need to show that 
\[t(Y; M, q) t(\pi_{(0)}; M, q) \ind_1(\pi) \alpha = 0.\]
Since $\pi_* \alpha \in \ov A_0(Y; M, q)$, it follows that $t(Y; M, q)
\pi_*\alpha = \pi_* (t(Y; M, q) \alpha) = 0$. But now by Lemma~\ref{zero
cycles pushforward again}, we conclude that $t_0(\pi, M) (\ind_1 \pi) t(Y,
M) \alpha = 0$ as desired.

Next, suppose that $\beta \in \ov A_0(Y; M, q)$. We need to show that 
\[t(X; M, q) \ind_0(\pi) \beta = 0.\] 
By Lemma~\ref{multiple lifts}, we may find $\gamma \in \ov A_0(X; M, q)$
with $\pi_* \gamma = (\ind_0 \pi)\beta$.  Since $t(X; M, q) \gamma = 0$,
the desired conclusion follows.
\end{proof}

\begin{rem} \label{corestriction base change}
Let $X$ be a variety and $x \in X$ a closed point with residue field $L =
k(x)$. Let $m \in M(L)$. Then the prime chain $(x, m)_X$ may be identified
with the pushforward of the corresponding prime chain $(x, m)_{X_L}$.

We can see this as follows. We may view $x$ as both a morphism $\Spec(L)
\to X$ or $\Spec(L) \to X_L$. To remove ambiguity, we will temporarily
denote the first by $x$ and the second by $x'$. We have a commutative
diagram:
\[\xymatrix{
	\Spec(L) \ar[r]^-{x'} \ar[rd]_x & X_L \ar[d]^p \\
	 & X.
}\]
Identifying $A_0(\Spec(L), M) = M(L)$, we have $x_*(m) = (x, m)_X$ and
$x'_*(m) = (x', m)_{X_L}$. It therefore follows from the commutativity of
the diagram that $p_*(x', m)_{X_L} = (x, m)_X$ as claimed.
\end{rem}

\begin{lem} \label{rat tor}
Suppose that $X$ and $Y$ are proper varieties over $k$ such that $X_{k(Y)}$
and $Y_{k(X)}$ are both rational varieties, then $A_0(X; M) \cong A_0(Y;
M)$.
\end{lem}
\begin{proof}
In this case, if we consider $Z = X \times Y$, we see that $Z$ is
birationally isomorphic to $X \times \mbb P^{\dim Y}$ and to $Y \times \mbb
P^{\dim X}$. The result therefore follows from
\cite[Cor.~RC.13]{KarMer:SNV} and \cite[Thm.~3.3(b)]{Ful:IT}.
\end{proof}

\begin{cor} \label{geometrically rational torsion}
Suppose that $X$ is geometrically rational. Then $\ov A_0(X; M)$ is torsion.
\end{cor}
\begin{proof}
This follows from the standard restriction-corestriction argument and the
fact that $\ov A_0(X_{\ov k}, M_{\ov k})$ is trivial by
Lemma~\ref{rat tor}.
\end{proof}

Following \cite[Def.~3.15]{Kra:ZC}, we say that a morphism $\pi: X \to Y$
of $k$-varieties has rational fibers if for every field extension $L/k$
and every $L$-point $y \in Y(L)$, the scheme theoretic fiber $X_y$ is
rational.
\begin{cor} \label{rational fibers}
Suppose that $\pi : X \to Y$ is morphism of proper
$k$-varieties with rational fibers. Then for any cycle module $M$, we have
$A_0(X; M, q) = A_0(Y; M, q)$.
\end{cor}
\begin{proof}
This follows from the fact that since the generic fiber of $\pi$ is
rational, we have $k(X)$ is totally transcendental over $k(Y)$. This implies
that $X$ is birationally isomorphic to $Y \times \mbb P^{\dim X}$, and the
result then follows from \cite[Cor.~RC.13]{KarMer:SNV} and
\cite[Thm.~3.3(b)]{Ful:IT}.
\end{proof}

Let $G$ be a $k$-linear algebraic group acting on a $k$-variety $X$.
Following \cite[p.~243]{HHK}, we say that $G$ \textit{acts transitively on
the points of $X$} if, for every field extension $L/k$, $G(L)$ acts
transitively on $X(L)$. 

\begin{prop} \label{homogeneous chow trivial}
Let $G$ be a $k$-rational linear algebraic group acting on proper
$k$-varieties $X$ and $Y$, and suppose $\pi: X \to Y$ is a $G$-equivariant
morphism. Suppose that $G$ acts transitively on the points of $Y$. Then for
$y \in Y(k)$, the inclusion:
\[i : X_y \to X\]
induces an surjective map $i_* : A_0(X_y, M, q) \to A_0(X, M, q)$.
\end{prop}
\begin{proof}
Suppose that we have a prime cycle of the form $(z, m)$ where $z \in
X_{(0)}$ is a closed point with residue field $L = k(z)$ and $m \in
M_q(L)$. It will suffice to show that $(z, m)$ is equivalent to a prime
cycle of the form $(x, n)$ for some $x$ with $\pi(x) = y$. As in
Remark~\ref{corestriction base change}, if we set $p_X : X_L \to X$
(respectively $p_Y : Y_L \to Y$) to be the natural projection maps, then
regarding $z$ also as a point of $X_L$, we have $(z, m)_X = (p_X)_*
(z,m)_{X_L}$.

Since $G$ acts transitively on the points of $Y$, we may find a $g \in
G(L)$ such that $g(\pi(z)) = y_L$. Since $G$ is $R$-trivial, we may find a
rational curve $\varphi : \PP^1_L \dashrightarrow G_L$ such that $\varphi(0)=
\id_G$, $\varphi(1) = g$. We may then define a new rational curve $\psi :
\PP^1_L \dashrightarrow X_L$ in $X_L$ by $\psi(t) = \varphi(t)(z)$, and by
the properness of $X_L$, $\psi$ extends to a morphism.

Consider the exact sequence obtained using \cite[Prop. 2.2]{Rost:CGC}:
\[C_1(\mbb P^1_L; M_L, q) \overset{d}\to C_0(\mbb
P^1_L; M_L, q) \overset{r_*}\to M(L), \]
where $r$ is the structure morphism for the $L$-scheme $\mbb P^1_L$.
Since the cycle $(0, m)_{\PP^1_L} - (1, m)_{\PP^1_L}$ is in the kernel of
$r_*$, we may choose $\alpha \in C_1(\PP^1_L; M, q)$ with $\bound_{\PP^1_L}
\alpha = (0, m)_{\PP^1_L} - (1, m)_{\PP^1_L}$. Applying $\psi_*$, we find
\begin{multline*}
(z, m)_{X_L} = \psi_*(0, m) \sim \psi_*(0, m) - \bound_{X_L}\psi_*
\alpha 
\sim 
\psi_*(0,m) - \psi_*\bound_{\PP^1_L} \alpha \\ \sim 
\psi_*\left((0,m) - \bound_{\PP^1_L} \alpha\right) \sim \psi_*(1, m) =
(g(z), m)_{X_L}.
\end{multline*}
Now, applying $(p_X)_*$ to each side, we find 
\[(z, m)_X \sim (g(z), c_{L/k} m)\]
But now, we note that we have $\pi(g(z)) = g(\pi(z)) = y$ which implies
that $(g(z), c_{L/k} m) \in i_*(Z_0(X; M, q))$.  as desired.
\end{proof}

\section{Isotropic Grassmannians}

Throughout the section, we fix a field $k$ of characteristic not $2$.
Suppose we have a regular quadratic form $\phi$ over $k$. We recall that the
\textit{splitting pattern} of $\phi$, introduced and studied in
\cite{Kneb:GS1,Kneb:GS2,HuRe:SPEQF}, is defined to be the collection of all
possible Witt indices which $\phi$ attains upon field extensions of $k$. We
write these as
\[ i(\phi) = i_0 < i_1 < \cdots < i_h = \left\lfloor \frac{\dim \phi}{2}
\right\rfloor, \]
and we call $h = h(\phi)$ the \textit{height} if the form $\phi$.

\begin{lem} \label{quadric fibers}
Let $(V, \phi)$ be a quadratic space and let $W < V$ be a totally isotropic
subspace. Choose a hyperbolic space $H$ of dimension $2 \dim W$ containing
$W$, and write $V = H \perp K$. Then we have a natural bijection of sets
\[\{ W < U < V \mid \text{$U$ totally isotropic}\} \leftrightarrow
\{P < K \mid P \text{ totally isotropic}\} 
\]
given by $U \mapsto U \cap K$ and $P \mapsto P + W$.
\end{lem}
\begin{proof}
Clearly both maps are well defined and the composition 
\[P \mapsto P+W \mapsto (P+W) \cap K = P\]
shows that the map $P \mapsto P+W$ is injective with the other map as a one
sided inverse. It therefore suffices to show that the map $P \mapsto P+W$
is surjective.

Suppose that $W < U < V$ with $U$ totally isotropic, and let $u \in U$. It
suffices to show that we may write $u = u' + w$ where $w \in W$ and $u' \in
U \cap K$. 

Since $U$ is totally isotropic and $W \subset U$, it follows
that $U \subset W^{\perp} = W \perp K$. We may therefore write $u = u' + w$
where $w \in W$ and $u' \in K$. But since $w \in W \subset U$ and $u \in
U$, it follows that $u' \in U$ and hence $u' \in K \cap U$ as desired.
\end{proof}
 
\begin{prop} \label{no torsion small index}
Let $\phi$ be a regular quadratic form, and suppose that $2 i(\phi) + 2 \leq \dim
\phi$. Then $\ov A_0(X_{i(\phi) + 1}, \mbb Z) = 0$.
\end{prop}
\begin{proof}
Set $X = X_{i(\phi) + 1}, Y = X_{i(\phi)}$ and let $\ms X = X_{i(\phi), i(\phi) + 1}$
be the variety of isotropic flags of dimension $i(\phi), i(\phi) + 1$. Note that
he statement of the Proposition is equivalent to the statement that $t_0(X,
\mbb Z) = \mbb Z$. We have natural projection maps:
\[\xymatrix{
	\ms X \ar[r]^{\pi_1} \ar[d]_{\pi_2} & Y \\
	X.
}\]
We may identify the fibers of $\pi_2$ over an isotropic plane $W < V$ with
the Grassmannian of $i(\phi)$-dimensional subspaces of $W$. In particular,
each of the fibers of $\pi_2$ is $R$-trivial. We therefore have $t_0(\pi_2,
\mbb Z) = \mbb Z$ and $\ind(\pi_2) = 1$. By Proposition~\ref{fiber
induction proposition}, we may conclude that $t_0(\ms X, \mbb Z) = t_0(X,
\mbb Z)$.

By the definition if $i(\phi)$, $Y$ has a rational point $y \in Y(k)$. By
Lemma~\ref{quadric fibers}, we may identify the fiber $\ms X_y$ with a
quadric hypersurface. By \cite{Swan} or \cite[Thm~8.8]{Kra:ZC}, we have
$\ov A_0(\ms X_y, \mbb Z) = 0$ and therefore by
Proposition~\ref{homogeneous chow trivial} using $G = SO(\phi)$, we
conclude $\ov A_0(\ms X, \mbb Z) = 0$, completing the proof.
\end{proof}

\begin{lem} \label{torsion no jump}
Let $\phi$ be a regular quadratic form over $k$, and suppose that $m$ is a
positive integer with $i_j < m \leq i_{j+1}$ for some $j < h(\phi)$. Then
for any cycle module $M$, we have $t_0(X_m(\phi); M, q) =
t_0(X_{i_{j+1}}(\phi); M, q)$.
\end{lem}
\begin{proof}
By Lemma~\ref{quadric fibers}, we may identify the fibers of the morphism
$X_{m, i_{j+1}}(\phi) \to X_m$ with an isotropic Grassmannian of a
quadratic form Witt equivalent to $\phi$, and by the
definition of the splitting pattern, each of these fibers is nonempty and
hence rational. We also note that each of the fibers of the map $X_{m,
i_{j+1}}(\phi) \to X_{i_{j+1}}(\phi)$ can be identified with a Grassmannian
$G(m, i_{j+1})$, and hence are also rational. The result therefore
follows from Corollary~\ref{rational fibers}.
\end{proof}

\begin{thm}
Let $\phi$ be a regular quadratic form over $k$, and suppose that $m$ is a
positive integer with $i_j < m \leq i_{j+1}$ for some $0 \leq j < h(\phi)$. Then
$\ov A_0(X_m(\phi), \mbb Z)$ is $2^j$-torsion.
\end{thm}
\begin{proof}
We prove this by induction on $j$, starting with $j = 0$. In this case, we
have $i_0(\phi) = i(\phi) = i$, and by Proposition~\ref{no torsion small
index}, we have $\ov A_0(X_{i + 1}, \mbb Z) = 0$. Consequently, $t(X_{i+1},
\mbb Z) = (1)$. By Lemma~\ref{torsion no jump}, $t(X_{i+1}, \mbb Z) =
t(X_{i_1}, \mbb Z) = t(X_m, \mbb Z)$, which implies $t(X_m, \mbb Z) = (1)$,
and $\ov A_0(X_m, \mbb Z) = 0$ as claimed.

For the induction step, suppose that $\ov A_0(X_{i_j}(\phi), \mbb Z)$ is
$2^{j-1}$-torsion. Considering the morphism $\pi: X_{i_j, i_j + 1} \to
X_{i_j}$, we find that the fibers are quadric hypersurfaces by
Lemma~\ref{quadric fibers}. In particular, $t(\pi, \mbb Z) = (1)$, and
$\ind(\pi) | 2$. We therefore have $2^j \in 2 t(X_{i_j}) \subset t(X_{i_j,
i_j + 1})$. But since $X_{i_j, i_j + 1} \to X_{i_j + 1}$ has fibers
isomorphic to Grassmannians $G(i_j, i_j + 1)$ which are therefore rational,
it follows from Corollary~\ref{rational fibers} that $t(X_{i_j + 1}, \mbb
Z) = t(X_{i_j, i_{j+1}})$, showing that $2^j \in t(X_{i_j + 1})$ as
claimed.
\end{proof}

\section{Generalized Severi-Brauer varieties}

Let $A$ be a central simple algebra of degree $n$. Given integers $1 \leq
n_1 < n_2 < \cdots < n_k < n$, we set $X_{n_1, \ldots, n_k}(A)$ to be the
Brauer-Severi flag variety parametrizing flags of right ideals $I_{1}
\subset \cdots \subset I_{k} \subset A$, where $\dim I_{i} = n n_i$ (see
\cite[Def.~4.1]{Kra:ZC} for a functorial description of this variety). For
a right ideal $I < A$, and for a positive integer $m$ with $m n < \dim I$,
we set $X_m(I)$ to be the variety of right ideals of $A$ of dimension $m n$
contained in $I$. We recall from \cite[Thm.~4.8]{Kra:ZC} that $X_m(I) \cong
X_m(D)$ for some central simple algebra $D$ of degree $\dim I$, Brauer
equivalent to $A$. It follows from Lemma~\ref{rat tor} that $A_0(X_{n_1,
\ldots, n_k}(A); M, q) = A_0(X_d(A); M, q)$ for any cycle module $M$, where
$d$ is the $\gcd$ of the integers $n_i$ and the index of $A$. In
particular, in understanding $t(X_{n_1, \ldots, n_k}(A))$ it suffices to
consider Brauer-Severi flag varieties of the form $X_d(A)$ with $d |
\ind(A)$. 

\begin{thm}
Let $A$ be a central simple algebra of period $2$ and index $2m$. Then for
$d | m$, we have $\ov A_0(X_{2d}(A), \mbb Z)$ is $m$-torsion.
\end{thm}
We note that since $\ind(A)$ and $\per(A)$ always have the same prime
factors, it follows that $m$ is a power of $2$.
\begin{proof}
Consider the natural projection morphisms
\[\xymatrix{
 & X_{2, 2d}(A) \ar[rd]^{\pi_1} \ar[ld]_{\pi_2} \\
X_{2d}(A) & & X_2(A).
}\]
To examine the fibers of $\pi_1$, we may switch to the opposite algebra
$A^{op}$ and consider the isomorphisms $X_{2, 2d}(A) = X_{2m - 2d, 2m -
2}$, $X_2(A) = X_{2m - 2}(A^{op})$, given by taking right ideals to their
left annihilators. In particular, we find that we may identify the fibers
of $\pi_1$ with varieties of the form $X_{2m-2d}(I) = X_{2m-2d}(D)$ for an
algebra $D$ of degree $2m - 2$ and index $2$. In particular, since they are
projective homogeneous varieties with rational points, the fibers of
$\pi_1$ are rational trivial, and hence $t(X_2(A), \mbb Z) = t(X_{2d, 2},
\mbb Z)$ by Corollary~\ref{rational fibers}.

On the other hand, we may identify the fibers of $\pi_2$ with varieties of
the form $X_2(D)$ where $D$ is a central simple algebra of degree $2m$. In
particular, $t(X_2(D), \mbb Z) = (1)$ by \cite[Thm.~7.3]{Kra:ZC} and
$\ind(X_2(D), \mbb Z) = m$ by \cite[Lem.~7.1]{Kra:ZC}. It therefore follows
from Proposition~\ref{fiber induction proposition}, that $\ov A_0(X_{2,
2m})(A)$ and hence $\ov A_0(X_{2d}(A))$ is $2^m$-torsion as desired.
\end{proof}

\def\cprime{$'$} \def\cprime{$'$} \def\cprime{$'$} \def\cprime{$'$}
  \def\cprime{$'$} \def\cprime{$'$} \def\cprime{$'$}
  \def\cftil#1{\ifmmode\setbox7\hbox{$\accent"5E#1$}\else
  \setbox7\hbox{\accent"5E#1}\penalty 10000\relax\fi\raise 1\ht7
  \hbox{\lower1.15ex\hbox to 1\wd7{\hss\accent"7E\hss}}\penalty 10000
  \hskip-1\wd7\penalty 10000\box7}

\end{document}